\documentclass[reqno,11pt]{amsart}
\usepackage{amsmath,amssymb,amsthm,mathrsfs}
\usepackage{epsfig,color}
\usepackage[latin1]{inputenc}
\usepackage[english]{babel}
\usepackage[numbers]{natbib}
\usepackage[colorlinks, citecolor=blue, linkcolor=red]{hyperref}

\numberwithin{equation}{section}

\def\00{{\bf 0}}

\def\op{\overline{\partial}_t}

\def\SS{\mathbb S}
\def\RR{\mathbb R}

\def\T{\mathcal{T}}

\newcommand{\eps}{{\varepsilon}}

\def\gt{\widetilde{g}}

\newtheorem*{theorem*}{Theorem}
\newtheorem{theorem}{Theorem}[section]
\newtheorem{lemma}[theorem]{Lemma}

\newtheorem{corollary}[theorem]{Corollary}
\newtheorem{remark}[theorem]{Remark}

\newtheorem{definition}[theorem]{Definition}

\begin{document}
    \title[A closure result for globally hyperbolic spacetimes]{A closure result for globally hyperbolic spacetimes}
  \date{}

\author{Giovanni Catino and Alberto Roncoroni}

\address{G. Catino, Dipartimento di Matematica, Politecnico di Milano, Piazza Leonardo da Vinci 32, 20133, Milano, Italy.}
\email{giovanni.catino@polimi.it}

\address{A. Roncoroni, Dipartimento di Matematica, Politecnico di Milano, Piazza Leonardo da Vinci 32, 20133, Milano, Italy.}
\email{alberto.roncoroni@polimi.it}

\begin{abstract}
In this paper we prove a closure result for globally hyperbolic spacetimes satisfying, at a certain time, natural assumptions on the deceleration, the pressure and the Hubble constant. The main tool that we use is a general  Bonnet-Myers type result.
\end{abstract}

\maketitle

\begin{center}


\medskip

\centerline{\bf AMS subject classification: 53C21, 83C05}

\end{center}

\

\

\section{Introduction}

An interesting and fascinating question in cosmology is the following: is the universe closed (i.e. compact and without boundary) or not and what topology does it have? The picture is completely clear in the context of the standard model of cosmology given by the Friedmann-Lemaitre-Robertson-Walker (FLRW) spacetime  solutions of the Einstein equation. Indeed, in this case, the topology of the three-dimensional spatial manifold is fixed \emph{a priori} since it has constant sectional curvature $K=-1,0,1$. In particular it is closed provided $K=1$. Hence,  in the FLRW cosmologies the closure question is trivial. We mention that a recent analysis in \cite{div} suggests a closed universe. Motivated by this and the recent results in \cite{gall} where the authors address these questions using a standard tool in differential geometry known as Bonnet-Myers theorem, in this paper we firstly prove a Bonnet-Myers type result and then we apply it to show a closure result for general globally hyperbolic spacetimes. 

\subsection{A Bonnet-Myers type result} 
We first recall that the classical Bonnet-Myers theorem states the following: given a complete Riemannian manifold $(M^n,g)$, $n\geq 2$, whose Ricci curvature satisfies 
\begin{equation}\label{Ric}
\mathrm{Ric}\geq (n-1)\lambda g\, , 
\end{equation}
for some $\lambda>0$. Then 
\begin{equation}\label{eq-d}
\mathrm{diam}(M^n,g)\leq \frac{\pi}{\sqrt{\lambda}}\, . 
\end{equation}
In particular $M^n$ is closed and has finite fundamental group (see e.g. \cite[Theorem 6.3.3]{Pet}). The Bonnet-Myers theorem has been investigated a lot by the Riemannian geometry community; in particular a similar result turns out to be true if one replaces the Ricci curvature by the so-called \emph{$m-$Bakry-\'Emery Ricci tensor:}
$$
\mathrm{Ric}^m_f:=\mathrm{Ric}+ \nabla^2 f-\frac{1}{m}df\otimes df \, ,
$$
where $m>0$ and $f:M\rightarrow\mathbb{R}$ is a smooth function called the \emph{potential}. Indeed, given a complete Riemannian manifold $(M^n,g)$, $n\geq 2$ whose $m-$Bakry-\'Emery Ricci tensor satisfies 
\begin{equation}\label{BER}
\mathrm{Ric}^m_f\geq \left(n+m-1\right)\lambda g\, , 
\end{equation}
for some $\lambda>0$, then \eqref{eq-d} holds (see \cite{qia} and also \cite{bq,l,ww}). If one takes $f=-\log u$, for some smooth and positive function $u:M\rightarrow\mathbb{R}$, \eqref{BER} reads as follows
$$
\mathrm{Ric}\geq \frac{\nabla^2 u}{u}+\left(\frac{1}{m}-1\right) \frac{du\otimes du}{u^2}+\left(n+m-1\right)\lambda g\,.
$$
It is well known that, if the Ricci  (or the $m$-Bakry-Emery Ricci) tensor is not uniformly positive, the closedness of the manifold is not guaranteed. However, our first result shows that this is not the case if the potential $u$ is a positive supersolution to a suitable elliptic PDE. More precisely, we prove the following:
\begin{theorem}\label{teoA}
Let $(M^n,g)$, $n\geq 2$, be a complete Riemannian manifold such that 
\begin{equation}\label{Ric}
\mathrm{Ric}\geq \alpha \frac{\nabla^2 u}{u} + \beta \frac{du\otimes du}{u^2} + \mathcal{Q}  \quad\text{ in } M\, ,
\end{equation}
where $\alpha,\beta\in\mathbb{R}$, $\mathcal{Q}$ is a symmetric two tensor and $u\in C^{\infty}(M)$ satisfies 
\begin{equation}\label{u}
u>0,\quad -\Delta u\geq V u + \gamma \frac{\vert\nabla u\vert^2}{u}\quad\text{ in } M \, ,
\end{equation}
where $\gamma\in\mathbb{R}$, $V\in C^{\infty}(M)$. Assume that, there exists $k\in\mathbb{R}$ such that 
\begin{equation}\label{F}
\mathcal{Q}+ kVg \geq (n-1)\lambda g \, ,
\end{equation}
for some $\lambda>0$,
\begin{equation}\label{hp1}
k\left(\gamma+1-\alpha\right)\geq 0
\end{equation}
and 
\begin{equation}\label{hp2}
\alpha+\beta+k(\gamma+1)-(n-1)\frac{k^2}{4}>0\,.
\end{equation}
Then $M^n$ is closed, has finite fundamental group and its diameter satisfies
$$
\mathrm{diam}(M^n,g)\leq \pi\sqrt{\frac{1}{\lambda}\left(1+\frac{\left[2\alpha-k(n-3)\right]^2}{4(n-1)\left[\alpha+\beta+k(\gamma+1)-(n-1)\frac{k^2}{4}\right]}\right)}\, .
$$
Finally, if in addition $V\geq 0$ and $\gamma\geq 0$, then $V\equiv 0$ on $M^n$.
\end{theorem}
Clearly, taking $u\equiv\text{const}$, $V\equiv 0$, $\mathcal{Q}=(n-1)\lambda g$ with the following parameters $\alpha=\gamma=k=0$ and $\beta=1$ this result recovers the classical Bonnet-Myers theorem. If $V$ has a sign, the condition \eqref{F} permits a negative lower bound on the tensor $\mathcal{Q}$.  Being the assumptions of the previous theorem very general, we expect that it can be used in different contexts. An immediate corollary is the following extension of a classical Cheng's result:

\begin{corollary}\label{c-cheng} Let $(M^n,g)$, $n\geq 2$, be a complete Riemannian manifold with $\mathrm{Ric}\geq -(n-1)g$. If there exists a positive solution $u\in C^{\infty}(M)$ of 
$$
-\Delta u \geq \mu\, u
$$
for some $\mu>0$, then 
$$
\mu\leq \frac{(n-1)^2}{4}.
$$
\end{corollary}
This corollary extends the well known upper bound for the first eigenvalue of $-\Delta$ obtained by Cheng \cite{che}. A comparison argument shows that the previous estimate follows by the aforementioned Cheng's result (see \cite[Chapter 9]{li}), however our proof does not rely on it (and actually gives an alternative approach on the problem).

We mention that an interesting aspect of the Bonnet-Myers theorem is the rigidity problem, i.e. if the equality hods in \eqref{eq-d} then $(M,g)$ must be isometric to the standard sphere. This problem has been investigated in \cite{che} and in \cite{Ruan} for the Ricci tensor and for the Bakry-\'Emery Ricci tensor, respectively. It would be interesting to investigate the analogue rigidity for Theorem \ref{teoA}.

The proof of Theorem \ref{teoA} is based on a conformal change of the metric $g$ via the function $u$ and is inspired by the proofs of \cite[Theorem 1]{enr} and of \cite[Theorem 1.2]{cmr}, in the context of stable constant mean curvature hypersurfaces and of finite index minimal hypersurfaces,  respectively. In this context we recover \cite[Corollary 1.3]{cmr} indeed by taking $u$ the stability function, $V=|A|^2$,  $\alpha=\beta=\gamma=0$, $k=\frac{n-1}{n}$ and $\mathcal{Q}=-A^2$. We refer to \cite{cmr} for further details.

\medskip

In the second part of the paper we apply Theorem \ref{teoA} to globally hyperbolic spacetimes.

\subsection{Application to General Relativity} A general solution to the \emph{Einstein equation}, usually called a \emph{spacetime},  is a smooth, connected, four-dimensional Lorentzian manifold $(X^4,\gamma)$ with signature $(-,+,+,+)$ satisfying 
$$
\mathrm{Ric}-\frac12\mathrm{R}\gamma=\mathcal{T}
$$
where $\gamma$ is the Lorentzian metric on $X^4$, $\mathrm{Ric}$, $\mathrm{R}$  denote the Ricci and the scalar curvature of $(X^4,\gamma)$ and $\mathcal{T}$ is a divergence-free, symmetric two tensor called the {\em total stress-energy tensor}. In local coordinates we have
\begin{equation}\label{Ein}
R_{\alpha\beta}-\frac{1}{2}R\gamma_{\alpha\beta}=\mathcal{T}_{\alpha\beta}\, , 
\end{equation}
for $\alpha,\beta=0,1,2,3$. Note that, in standard notations, the tensor $\mathcal{T}$ can be decomposed in the following way 
$$
\mathcal{T}_{\alpha\beta}= T_{\alpha\beta}-\Lambda \gamma_{\alpha\beta} \, , 
$$
where $T$ is the \emph{stress-energy tensor},  and $\Lambda$ is the {\em cosmological constant}. To avoid undesirable pathologies of the causal structure of the spacetime it is customary to postulate the existence of a smooth spacelike Cauchy hypersurface $M^3\subset X^4$,  i.e.  a smooth hypersurface $M^3$ with the property that any inextendible causal curve intersects it at precisely one point.  Spacetimes $(X^4,\gamma)$ with this property are called \emph{globally hyperbolic} and are,  in particular, stable causal, i.e. they allow the existence of a globally defined differentiable function $t$ whose gradient is time-like (see \cite{bersan} and also \cite{c,bc}).  We call $t$ a time function and the foliation given by its level surfaces a $t-$foliation.  Topologically,  a spacetime foliated by the level surfaces of a time function is diffeomorphic to a product manifold $I\times M^3$ where $I\subset \RR$ and $M^3$ is a three-dimensional smooth manifold without boundary, usually called the \emph{slice} (see \cite{cho}). Relative to this parametrization the spacetime manifold $(X^4,\gamma)$ takes the form 
$$
(X^4,\gamma)=(I\times M^3, - N^2(t,x) dt^2+g_{ij}(t,x)dx^i dx^j)\, , 
$$
where $t=x^0\in I$ and $x=(x^1,x^2,x^3)$ are arbitrary coordinates on the slice and $i,j=1,2,3$. The function $N(t,x)$ is called the \emph{lapse function} of the foliation and $g_{ij}$ its first fundamental form.  We will assume that the metric $g$ on $M^3$ is complete. We will denote by
$$
M_t=\lbrace t\rbrace \times M^3\, 
$$
the \emph{leaves} of the foliation and by $\nu$ the future directed unit normal given by 
$$
\nu=\frac{\partial_t}{N}:=\op\, .
$$
The {\em second fundamental} form $h$ of the foliation is given by 
\begin{equation}\label{fond}
h_{ij}=-\dfrac{1}{2N}\partial_t g_{ij} \, ,
\end{equation}
and we denote with $H$ the {\em mean curvature} of the foliation, i.e.  $H=g^{ij}h_{ij}$.

In this notation, it is useful to introduce the following well known terminology concerning the stress-energy tensor $T$:
$$
T_{\nu\nu}=:\rho,\qquad T_{ii}=:p_i,\quad i=1,2,3,
$$
where $\rho$ and $p_i$ are called the {\em energy density} and the {\em principal pressures}, respectively (see e.g. \cite[Chapter 4.3]{he}). To be precise, this terminology, is used when $\left\{-\nu, \frac{\partial}{\partial x^1}, \frac{\partial}{\partial x^2}, \frac{\partial}{\partial x^3}\right\}$ is a orthonormal basis of eigenvectors of $T$, for this reason we will assume that $T$ is diagonalizable.

As already observed in \cite[Introduction]{gall}, in the globally hyperbolic setting Gauss equations satisfied by the slice $M^3\subset X^4$ provide a formula for a weighted version of the Ricci tensor, the so-called Bakry-\'Emery Ricci tensor, which relates it to the matter content ($\mathcal{T}$) and the intrinsic geometry ($h$). This allows us to apply Theorem \ref{teoA} in this setting under suitable assumptions on the leave of the foliation $M_{t_0}$ at some $t_0\in I$. We refer to Section \ref{proof} for further details. 

Before presenting our results, we recall some well known {\em parameters} (actually functions) which describe the kinematic of cosmological expansion in a FLRW spacetime. In this setting the metric of the spacetime $(X^4=I\times M^3,\gamma)$ is given by
\begin{equation}\label{FLRW}
\gamma = -dt^2+a(t)^2 g^{K}
\end{equation}
where $a$ is a positive smooth function and $g^K$ is a metric of constant sectional curvature equal to $K$ on $M^3$. In the previous notation, we have $N\equiv 1$, $\op=\partial_t$ and $g=a(t)^2 g^K$. The {\em Hubble parameter} and the {\em deceleration parameter} are defined as follows
$$
\mathcal{H}(t):=\frac{a'(t)}{a(t)},\qquad q(t):=-\frac{a''(t)a(t)}{a'(t)^2}.
$$
Note that the Hubble parameter and the mean curvature are related by the following $\mathcal{H}=-\frac13 H$. Moreover, given a point $x\in M$ and a tangent vector $V\in T_xM$, we have
$$
|V|:=\sqrt{g(V,V)}= a(t)\sqrt{g^K(V,V)}\,, 
$$ 
and therefore
$$
|V|'' |V| = -q (|V|')^2.
$$
This simple relation on the FLRW cosmological model motivates the following:

\begin{definition}
Let $(X^4,\gamma)=(I\times M^3, -N^2 dt^2+g)$ be a globally hyperbolic spacetime. We define

\begin{itemize}

\item the {\em deceleration parameter} of the spacetime as
$$
q(t):=\sup \left\{q\in\RR: \overline{\partial}^2_{tt} \vert V\vert \vert V\vert_{\left.\right|_{(t,x)}}\leq -q \left(\op \vert V\vert\right)^2_{\left.\right|_{(t,x)}}, \forall x\in M, \forall V\in T\mathcal{A}_{(t,x)} \right\};
$$ 
where $\mathcal{A}_{(t,x)}\subset X^4$ is a neighbourhood of $(t,x)$.

\item The {\em Hubble parameter} of the spacetime as
$$\mathcal{H}(t):=\sqrt{\inf_{M_{t}} H^2}.$$

\item The {\em pressure parameter} of the spacetime as
$$
\mathcal{P}(t):=\max_{i=1,2,3} \left( \sup_{M_t} p_i\right)\,.
$$
\end{itemize}
\end{definition}

\begin{remark}\label{length} We recall that, given a smooth space curve $\sigma:J\subset\RR\to M^3$, the {\em length functional} is defined as
$$
L(\sigma):=\int_{\sigma}|\dot\sigma|\,.
$$
We observe that the condition $q(t_0)= 0$ for some $t_0\in I$ implies that the length of all space curves is concave in the (normalized) time, i.e. for all curves $\sigma:J\subset\RR\to M^3$ we have 
$$
\overline{\partial}^2_{tt}\mathrm{L}(\sigma)_{\left.\right|_{t_0}}= \int_{\sigma}\overline{\partial}^2_{tt} \vert \dot{\sigma}\vert_{\left.\right|_{t_0}} \leq 0\, . 
$$
More in general, if $q(t_0)\geq 0$ for some $t_o\in I$, by using H\"older inequality, we have 
$$
\overline{\partial}^2_{tt}\mathrm{L}(\sigma)_{\left.\right|_{t_0}}= \int_{\sigma}\overline{\partial}^2_{tt} \vert \dot{\sigma}\vert_{\left.\right|_{t_0}}\leq -q(t_0) \int_\sigma \frac{\left(\overline{\partial}_{t} \vert \dot{\sigma}\vert_{\left.\right|_{t_0}}\right)^2}{\vert \dot{\sigma}\vert}\leq -q(t_0)\left( \int_\sigma \overline{\partial}_{t} \vert \dot{\sigma}\vert_{\left.\right|_{t_0}}\right)^2 \left(\int_\sigma \vert\dot{\sigma}\vert \right)^{-1}
$$
i.e. at $t_0$
$$
\mathrm{L}(\sigma)\, \overline{\partial}^2_{tt}\mathrm{L}(\sigma)\leq -q(t_0) \, \overline{\partial}_{t}\mathrm{L}(\sigma)^2\,.
$$
\end{remark}
It can be shown that (see e.g. \cite[Chapter 5]{he}) a FLRW metric \eqref{FLRW} satisfies the Einstein equation \eqref{Ein} if and only if
$$
\begin{cases}
\frac{a'(t)^2}{a(t)^2}=\frac13 \T_{\nu\nu}-\frac{K}{a(t)^2} \\
\frac{a''(t)}{a(t)}=-\frac12 \T_{ii}-\frac16 \T_{\nu\nu}
\end{cases}
$$
for all $i=1,2,3$. Therefore a simple computation shows that
$$
q = \frac12+\frac{a^2 \T_{ii}+K}{2a^2\mathcal{H}^2} = \frac12+\frac{a^2 (p_i-\Lambda)+K}{2a^2\mathcal{H}^2}
$$
for all $i=1,2,3$, provided $\mathcal{H}^2>0$. In particular, we have that
\begin{equation}\label{flrw2}
q(t_0)>\frac12,\quad p_i(t_0)\leq \Lambda\quad\text{and}\quad \mathcal{H}(t_0)^2>0\qquad\Longrightarrow\qquad K>0,
\end{equation}
that is, the spacetime is closed.

This simple observation motivates the following result which applies to general globally hyperbolic spacetimes.

\begin{theorem}\label{t-main}
Let $(X,\gamma)=(I\times M^3, - N^2 dt^2+g)$ be a globally hyperbolic spacetime such that $g$ is complete and
$$
a)\,\, q(t_0)>\frac{1}{2}\qquad b)\,\, \mathcal{P}(t_0)\leq \Lambda\qquad c)\,\, \mathcal{H}(t_0)^2>0\,,
$$
for some $t_0\in I$. Then $M^3$ is compact.  In addition, $M^3$ is diffeomorphic to a quotient of $\mathbb S^3$ and its diameter satisfies
$$
\mathrm{diam}(M^3,g_{\left.\right|_{t_0}})\leq \frac{\pi}{\mathcal{H}(t_0)}\sqrt{\frac{8(10q+4)}{3(2q-1)(2q+1)}}\,.
$$
\end{theorem}
We point out that $a), b), c)$ are the natural generalizations of the assumption in \eqref{flrw2} that guarantee the closure result for FLRW spacetime. 

An important case in which our result applies is in the context of a \emph{perfect fluid spacetime}. In this setting we have a globally hyperbolic spacetime with the tensor $\mathcal{T}$ satisfying
\begin{equation}\label{e-pf}
\T_{\nu\nu}= \rho+\Lambda \, , \quad \T_{\nu i}=0 \,, \quad \T_{ij}= \left(p-\Lambda\right)g_{ij}\, , 
\end{equation}
where $\rho=\rho(t,x)$ is the \emph{energy density} and $p=p(t,x)$ is the \emph{pressure} of the fluid (in the previous notation $p_i=p$ for all $i=1,2,3$). It is well known that the pressure of perfect fluids depends only on time, $p=p(t)=\mathcal{P}(t)$. This follows from the Bianchi identity applied to \eqref{Ein} (see e.g. \cite[equation (5.2)]{bc}). With the previous notations we have the following 

\begin{corollary}\label{cor}
Every perfect fluid spacetimes with $q(t_0)>1/2$, $p(t_0)\leq \Lambda$ and $\mathcal{H}(t_0)^2>0$, for some $t_0\in I$, has compact time slices. 
\end{corollary}

 In general, we can assume that the pressure and the energy density satisfy, for some function $\omega=\omega(t,x)$, the following \emph{equation of state}:
$$
p=\omega \rho.
$$
Special cases of interest in cosmology are when $\omega=\omega(t,x)$ is constant; in particular we can distinguish the following cases: \emph{radiation dominated} when $\omega=\frac{1}{3}$,  \emph{matter dominated} when $\omega=0$ and \emph{vacuum energy dominated} when $\omega=-1$. The assumption on the pressure in Corollary \ref{cor} follows if $\omega\leq 0$ and $\Lambda\geq 0$.

\medskip

Finally, we deal also with energy-decelerating spacetimes. Indeed, it is natural to consider, instead of the length functional, the {\em energy functional} of space curves:
$$
E(\sigma):=\int_\sigma |\dot \sigma|^2\,.
$$
In analogy to the definition of the deceleration parameter and to what we said in Remark \ref{length}, we introduce the following:

\begin{definition}
Let $(X^4,\gamma)=(I\times M^3, -N^2 dt^2+g)$ be a globally hyperbolic spacetime and let $t_0\in I$. We say that $M_{t_0}$ is {\em energy--decelerating} in all directions if the energy of all space curves is concave in the proper time, i.e. for all curves $\sigma:J\subset\RR\to M^3$ we have 
$$
\overline{\partial}^2_{tt} \text{E}(\sigma)_{\left.\right|_{t_0}}:=\int_{\sigma} \overline{\partial}^2_{tt}|\dot\sigma|^2_{\left.\right|_{t_0}} \leq 0,
$$
where $\vert \cdot\vert:=\sqrt{g(\cdot,\cdot)}$.
\end{definition}

Under this condition, we prove the following:

\begin{theorem}\label{t-main2}
Let $(X,\gamma)=(I\times M^3, - N^2 dt^2+g)$ be a globally hyperbolic spacetime with total stress-energy tensor $\T$. Assume that $g$ is complete and there exists $t_0\in I$ such that, 
\begin{itemize}
\item[$a)$] $M_{t_0}$ is {\em energy--decelerating} in all directions;
\item[$b)$] $\T_{ii}-k\T_{\nu\nu}-\frac{1+k}{2}\mathrm{tr}(\T)\geq 0$ on $M_{t_0}$, for all $i=1,2,3$ and for some $$0.59\simeq\frac{\sqrt{43}-3}{6}<k<1+\sqrt{3};$$
\item[$c)$] $\mathcal{H}(t_0)^2>0$.
\end{itemize}
Then $M^3$ is compact.  In addition, $M^3$ is diffeomorphic to a quotient of $\mathbb S^3$ and its diameter satisfies
$$
\mathrm{diam}(M^3,g_{\left.\right|_{t_0}})\leq \frac{6\pi}{\mathcal{H}(t_0)}\sqrt{\frac{(3+2k-k^2)(4+3k)}{(2+2k-k^2)(18k^2+18k-17)}}\,.
$$
\end{theorem}

In particular, in the context of perfect fluid spacetimes, an analogous of Corollary \ref{cor} can be proved.

%
%
%
%

\

\section{A Bonnet-Myers type result: proof of Theorem \ref{teoA} and Corollary \ref{c-cheng}}\label{BM}

\begin{proof}[Proof of Theorem \ref{teoA}] Let $k>0$ and consider the conformal metric
$$
\gt = u^{2k} g.
$$
Given a reference point $o\in M$ and $\rho>0$ such that the geodesic ball (of $g$) $B_{\rho}(o)$ is contained in the interior of $M$; we want to construct a $\gt-$minimizing geodesic $\widetilde\gamma$ in $B_{\rho}(o)$, joining $o$ to $\partial B_{\rho}(o)$. In order to do this we consider $u_{\rho}:=u+\eta_{\rho}$,  where $\eta_{\rho}$ is a smooth function such that $\eta_\rho\equiv1$ in $B_{\rho+1}(o)^c$ and $\eta_\rho\equiv 0$ in $B_{\rho}(o)$. Since $u_\rho$ is uniformly bounded below away from zero, the metric
$$\gt_\rho=u_\rho^{2k}g$$
is complete, and thus there exists a $\gt_\rho-$minimizing geodesic connecting $o$ to $\partial B_{\rho}(o)$. Now,  by compactness of $\partial B_{\rho}(o)$ we can define $\widetilde\gamma$ to be the shortest among all the $\gt_{\rho}-$minimizing geodesics constructed previously connecting $o$ to $\partial B_{\rho}(o)$. We remark that by construction $\widetilde\gamma$ is contained in $B_{\rho}(o)$, since if $\widetilde\gamma$ would escape from  $B_{\rho}(o)$ then there would be at least another point on $\partial B_{\rho}(o)$, thus contradicting the previous construction. Finally, since $u_{\rho}=u$ in $B_{\rho}$ we have that $\widetilde\gamma$ is $\gt-$minimizing (this construction appeared in \cite{fis}).

Let $l$ be the $g-$length of $\widetilde\gamma$. Then $l\geq\rho$ and in order to prove the theorem it is enough to show that 
$$
l\leq C
$$
for some $C=C(n,\alpha,\beta, \gamma,\lambda,k)>0$. 

Let $s$ and $\tilde{s}$ be the arc lengths with respect to the metric $g$ and $\tilde{g}$, respectively. We denote with $R$ and $\tilde{R}$ the curvature tensors of $M$ with respect to $g$ and to $\gt$. We choose a basis $\left\lbrace \tilde{e_1}=\frac{\partial\widetilde\gamma}{\partial\tilde s}, \tilde{e}_2,\dots, \tilde{e}_n \right\rbrace$  orthonormal for the metric $\gt$ such that $\tilde{e}_2,\dots, \tilde{e_n}$ are parallel along $\widetilde\gamma$. The basis $\left\lbrace e_1=\frac{\partial\widetilde\gamma}{\partial s}, e_2=u^k\tilde{e}_2,\dots, e_n=u^k\tilde{e}_n \right\rbrace$ is orthonormal for the metric $g$.  Denote by $R_{11}$ and $\tilde{R}_{11}$ the Ricci curvatures in the direction of $e_1$ for the metric $g$ and $\tilde{g}$, respectively. 

Since $\widetilde\gamma$ is $\gt-$minimizing, by the second variation formula, one has 
\begin{equation}\label{2nd}
\int_{0}^{\tilde{l}}\left[ (n-1)(\varphi_{\tilde{s}})^2 - \tilde{R}_{11}\varphi^2\right]\, d\tilde{s}\geq 0\, ,
\end{equation}
for any smooth function $\varphi$ such that $\varphi(0)=\varphi(\tilde{l})=0$, where $\tilde{l}$ denotes the $\gt$-length of $\widetilde\gamma$.

As proved in \cite[Appendix]{enr}, 
\begin{equation}\label{tR11}
 \tilde{R}_{11}=u^{-2k}\left\lbrace R_{11}- k(n-2)(\ln u)_{ss}- k \frac{\Delta u}{u}+ k\frac{\vert\nabla u\vert^2}{u^2} \right\rbrace \, .
\end{equation}
From \eqref{Ric} we obtain
\begin{equation}\label{R11_1}
R_{11}\geq \alpha \frac{\nabla^2_{11} u}{u} + \beta \frac{u_s^2}{u^2} + \mathcal{Q}_{11}\, ,
\end{equation}
where $\nabla^2_{11} u$ and $\mathcal{Q}_{11}$ denotes the Hessian of $u$ and the tensor $\mathcal Q$ in the direction $e_1$, respectively.  Actually, \eqref{R11_1} can be rewritten in the following way
\begin{equation}\label{R11_2}
R_{11}\geq \alpha\nabla^2_{11} (\ln u)+(\alpha+\beta)(\ln u)_s^2 + \mathcal{Q}_{11}  \, . 
\end{equation}
From \cite[Formula (13)]{enr} we deduce that  
$$
\nabla^2_{11} (\ln u)=\nabla^2 (\ln u)\left(\frac{\partial\widetilde\gamma}{\partial s},\frac{\partial\widetilde\gamma}{\partial s}\right)=(\ln u)_{ss}-\left(\nabla_{\frac{\partial\widetilde\gamma}{\partial s}}\frac{\partial\widetilde\gamma}{\partial s}\right)\ln u=(\ln u)_{ss}-k\vert(\nabla \ln u)^{\perp}\vert^2\, , 
$$
hence, \eqref{R11_2} becomes 
\begin{equation}\label{R11_3}
R_{11}\geq  \alpha(\ln u)_{ss}- \alpha k \vert (\nabla\ln u)^{\perp}\vert^2 +(\alpha+\beta)(\ln u)_s^2 +\mathcal{Q}_{11}\, .
\end{equation}
Plugging \eqref{R11_3} in \eqref{tR11} we have 
$$
 \tilde{R}_{11}\geq  u^{-2k}\left\lbrace \left[\alpha -k(n-2)\right] (\ln u)_{ss}- \alpha k \vert \nabla(\ln u)^{\perp}\vert^2+(\alpha+\beta)(\ln u)_s^2 + \mathcal{Q}_{11} - 
 k \frac{\Delta u}{u}+ k\frac{\vert\nabla u\vert^2}{u^2} \right\rbrace \,. 
$$
Furthermore, from \eqref{u} we have 
\begin{align*}
 \tilde{R}_{11}\geq& u^{-2k}\left\lbrace \left[\alpha -k(n-2)\right] (\ln u)_{ss}- \alpha k \vert \nabla(\ln u)^{\perp}\vert^2+(\alpha+\beta)(\ln u)_s^2 + \mathcal{Q}_{11} +  
 k V+ k(\gamma+1)\frac{\vert\nabla u\vert^2}{u^2} \right\rbrace \\
\geq&u^{-2k}\left\lbrace \left[\alpha -k(n-2)\right] (\ln u)_{ss}- \alpha k \vert \nabla(\ln u)^{\perp}\vert^2+(\alpha+\beta)(\ln u)_s^2 +(n-1)\lambda\right.\\&\qquad\left.+ k(\gamma+1) \left[ (\ln u)_s^2 + \vert \nabla(\ln u)^\perp \vert^2\right] \right\rbrace \\
\geq& u^{-2k}\left\lbrace \left[\alpha -k(n-2)\right] (\ln u)_{ss}+\left[\alpha+\beta+k(\gamma+1)\right](\ln u)_s^2 +(n-1)\lambda \right\rbrace \, , 
\end{align*}
where we used \eqref{F} and \eqref{hp1}. Plugging this information in \eqref{2nd} we get
\begin{align*}
(n-1)\int_{0}^{l}(\varphi_{s})^2u^{-k}\, ds \geq & \left[\alpha -k(n-2)\right] \int_{0}^{l}\varphi^2 u^{-k}(\ln u)_{ss}\, ds \\ 
&+\left[\alpha+\beta+k(\gamma+1)\right] \int_{0}^{l}\varphi^2 u^{-k}(\ln u)_s^2 \, ds  \\
&+ (n-1)\lambda \int_{0}^{l}\varphi^2 u^{-k}\, ds\, ,
\end{align*}
since $l$ is the $g$-length of $\widetilde\gamma$. Integrating by parts we obtain  
\begin{align}\label{2nd_bis}
(n-1)\int_{0}^{l}(\varphi_{s})^2u^{-k}\, ds \geq &-2 \left[\alpha -k(n-2)\right] \int_{0}^{l}\varphi\varphi_s u^{-k-1}u_s\, ds\nonumber \\ 
&+\left[\alpha+\beta+k(\gamma+1+\alpha)-(n-2)k^2\right] \int_{0}^{l}\varphi^2 u^{-k-2}u_s^2 \, ds \nonumber \\
&+ (n-1)\lambda \int_{0}^{l}\varphi^2 u^{-k}\, ds\,, 
\end{align}
 for any smooth function $\varphi$ such that $\varphi(0)=\varphi(l)=0$. By choosing 
 $$
 \varphi=u^{\frac{k}{2}}\psi \, ,
 $$
where $\psi$ is a smooth function such that $\psi(0)=\psi(l)=0$, in \eqref{2nd_bis} we conclude 
\begin{align}\label{2nd_ter}
(n-1)\int_{0}^{l}(\psi_{s})^2\, ds \geq &\left\lbrace-2 \left[\alpha -k(n-2)\right]-k(n-1)\right\rbrace \int_{0}^{l}\psi\psi_s u^{-1}u_s\, ds\nonumber \\ 
&+\left[\alpha+\beta+k(\gamma+1)-(n-1)\frac{k^2}{4}\right] \int_{0}^{l}\psi^2 u^{-2}u_s^2 \, ds \nonumber \\
&+ (n-1)\lambda\int_{0}^{l}\psi^2 \, ds\,.
\end{align}
Since \eqref{hp2} is in force we can use the fact that $a^2+b^2\geq-2ab$, with 
$$
a=\sqrt{\alpha+\beta+k(\gamma+1)-(n-1)\frac{k^2}{4}}\,  \psi \frac{u_s}{u}
$$
and 
$$
\quad b=\frac{-2 \left[\alpha -k(n-2)\right]-k(n-1)}{2\sqrt{\alpha+\beta+k(\gamma+1)-(n-1)\frac{k^2}{4}}}\, \psi_s\, , 
$$
to obtain the following 
\begin{multline*}
\left[\alpha+\beta+k(\gamma+1)-(n-1)\frac{k^2}{4}\right]\psi^2 u^{-2}u_s^2 + \frac{\left\lbrace-2 \left[\alpha -k(n-2)\right]-k(n-1)\right\rbrace}{4\left[\alpha+\beta+k(\gamma+1)-(n-1)\frac{k^2}{4}\right]}\, (\psi_s)^2\\ \geq -\left\lbrace-2 \left[\alpha -k(n-2)\right]-k(n-1)\right\rbrace \psi\psi_s u^{-1}u_s\, . 
\end{multline*}
Hence, \eqref{2nd_ter} can be rewritten in the following way
\begin{equation}\label{eq_fin_BM}
A\int_{0}^{l}(\psi_s)^2 \, ds \geq B \int_0^{l}\psi^2 \, ds \, ,
\end{equation}
where 
$$
A=n-1+\dfrac{\left[2\alpha-k(n-3)\right]^2}{4\left[\alpha+\beta+k(\gamma+1)-(n-1)\frac{k^2}{4}\right]} \quad \text{and} \quad B=(n-1)\lambda\, .
$$  
In particular, from \eqref{eq_fin_BM} we immediately get $BA^{-1}\leq \lambda_1$, where $\lambda_1=\pi^2l^{-2}$ is the first Dirichlet eigenvalue of $d^2/ds^2$ on the interval $[0,l]$. Hence 
\begin{equation}\label{concl_BM}
l \leq \pi\sqrt{\frac{A}{B}}=\pi\sqrt{\frac{1}{\lambda}\left(1+\frac{\left[2\alpha-k(n-3)\right]^2}{4(n-1)\left[\alpha+\beta+k(\gamma+1)-(n-1)\frac{k^2}{4}\right]}\right)}\, . 
\end{equation}
Therefore $(M^n,g)$ must be compact. Moreover, by applying the same strategy to the universal cover of $M^n$, we have that the fundamental group of $M^n$ must be finite. Finally,  if in addition $V\geq 0$ and $\gamma\geq 0$, then integrating \eqref{u} over $M^n$ we obtain $V\equiv 0$ and this concludes the proof of Theorem \ref{teoA}.
\end{proof}

\begin{remark} An alternative end of the previous proof can be performed arguing as follows: integrating by parts \eqref{eq_fin_BM} we obtain 
\begin{equation}\label{final}
\int_{0}^{l} \left(A\psi\psi_{ss} + B\psi^2\right) \, ds\leq 0 \, ,
\end{equation}
for all smooth functions $\psi$ such that $\psi(0)=\psi(l)=0$. By taking 
$$
\psi(s)=\sin\left(\frac{\pi s}{l}\right)\, , \quad \text{ for $s\in [0,l]$}\, , 
$$
in \eqref{final} we get 
$$
\left( B-\dfrac{A\pi^2}{l^2}\right) \int_{0}^{l} \sin^2\left(\frac{\pi s}{l}\right)\, ds \leq 0\, , 
$$
i.e. \eqref{concl_BM} and the conclusion follows.
\end{remark}

\medskip

\begin{proof}[Proof of Corollary \ref{c-cheng}]
Let $(M^n,g)$, $n\geq 2$, be a complete Riemannian manifold with $\mathrm{Ric}\geq -(n-1)g$ and let $u\in C^{\infty}(M)$ be a positive solution of 
$$
-\Delta u \geq \mu\, u
$$
for some $\mu>0$. In the notation of Theorem \ref{teoA}, we have 
$$
\alpha=\beta=\gamma=0,\quad \mathcal{Q}=-(n-1)g, \quad V=\mu.
$$ 
Therefore \eqref{F}-\eqref{hp1}-\eqref{hp2} read as
\begin{equation}\label{eqeq}
-(n-1)+k\mu\geq (n-1)\lambda,\quad k\geq 0,\quad k\left(1-\frac{n-1}{4}k\right)>0
\end{equation}
for some $\lambda>0$. By contradiction, suppose that $$\mu>\frac{(n-1)^2}{4}.$$ Taking $k=\frac{4}{n-1}-\eps$ for some $\eps>0$ small enough, we have that the conditions \eqref{eqeq} are satisfied and therefore $M^n$ must be closed. This contradicts the fact that $u>0$ and satisfies $-\Delta u \geq \mu\, u$.
\end{proof}

\

\section{Application to General Relativity} \label{proof}


In this section we prove Theorem \ref{t-main} and Theorem \ref{t-main2} together with Corollary \ref{cor}. We need the following

\begin{lemma}\label{l-dec} Let $(X^4,\gamma)=(I\times M^3, -N^2 dt^2+g)$ be a globally hyperbolic spacetime. Then  

\begin{itemize}

\item for all tangent vector $V$ one has
$$
\frac{1}{2N^2}\left(\partial^2_{tt} g_{ij} V^i V^j+2h_{ij}V^i V^j\partial_t N\right)\leq [1-q(t)]\frac{|h_{ij}V^i V^j|^2}{g(V,V)}\qquad \forall t\in I\, ;
$$

\item $M_{t_0}$ is {\em energy--decelerating} in all directions if and only if
$$
\partial^2_{tt} g_{ij} +2h_{ij}\partial_t N\leq 0\, ,
$$
in the sense of quadratic forms.
\end{itemize}
\end{lemma}
\begin{proof}
By using \eqref{fond} we obtain
\begin{align*}
\overline{\partial}^2_{tt} \sqrt{g(V,V)} &= \overline{\partial}^2_{tt} \sqrt{g_{ij}V^i V^j} = \frac{1}{N} \partial_t \left(\op \sqrt{g_{ij}V^i V^j}\right)\\
&=\frac{1}{N} \partial_t \left[\frac{\partial_t g_{ij}V^i V^j}{2 \sqrt{g(V,V)} N}\right]\\
&= \frac{\partial^2_{tt} g_{ij} V^i V^j}{2 \sqrt{g(V,V)} N^2}-\frac{\partial_t g_{ij}V^i V^j\partial_t N}{2\sqrt{g(V,V)}N^3}-\frac{|\partial_t g_{ij}V^i V^j|^2}{4g(V,V)^{\frac32}N^2}\\
&= \frac{\partial^2_{tt} g_{ij} V^i V^j}{2 \sqrt{g(V,V)} N^2}+\frac{h_{ij}V^i V^j\partial_t N}{\sqrt{g(V,V)}N^2}-\frac{|h_{ij}V^i V^j|^2}{g(V,V)^{\frac32}}\, .
\end{align*}
Thus from the definition of the deceleration parameter we have, for all $V$,
$$
\overline{\partial}^2_{tt} \sqrt{g(V,V)}\leq -q(t) \frac{\left(\overline{\partial}_{t} \sqrt{g(V,V)}\right)^2}{\sqrt{g(V,V)}}=-q(t)\frac{|\partial_t g_{ij}V^i V^j|^2}{4g(V,V)^{\frac32}N^2}=-q(t)\frac{|h_{ij}V^i V^j|^2}{g(V,V)^{\frac32}}\,, 
$$
and therefore, from the definition of $\op$,
$$
\frac{1}{2N^2}\left(\partial^2_{tt} g_{ij} V^i V^j+2h_{ij}V^i V^j\partial_t N\right)\leq [1-q(t)]\frac{|h_{ij}V^i V^j|^2}{g(V,V)}\, ,
$$
where we used \eqref{fond}.

\medskip

\noindent To prove the second estimate, we observe that, by continuity, the condition 
$$
\overline{\partial}^2_{tt} \,\mathrm{E}(\sigma)_{\left.\right|_{t_0}}  \leq 0\quad\forall \sigma:J\subset\RR\to M^3
$$
is equivalent to 
$$
\overline{\partial}^2_{tt} g_{ij}\,_{\left.\right|_{(t_0,x)}} \leq 0
$$
in the sense of quadratic forms. A computation similar to the one above gives the result.
\end{proof}

\begin{proof}[Proof of Theorem \ref{t-main} and Theorem \ref{t-main2}]
By classical formulas given by the immersion of $M^3\hookrightarrow X^4$ (see e.g. \cite{c} and also \cite[Formula (3.4)]{bc}), the second fundamental form, the lapse function $N$ and the Ricci curvature of $(M^3,g)$ are related to the Riemann tensor of $(X^4,\gamma)$. More precisely, the following equations holds 
$$
\begin{cases}
N\left(\T_{ij}-\frac12\T  g_{ij}\right) = N R_{ij} -\partial_t h_{ij} +N H h_{ij}-2N h_{il}h_{jl} -\nabla^2_{ij} N\\
N\left(\T_{\nu\nu}+\frac12\T\right)  = \partial_t H- N |h|^2+\Delta N\,,
\end{cases}
$$
i.e. 
\begin{equation}
\begin{cases}
R_{ij} =\frac{\nabla^2_{ij}N}{N} + \frac{\partial_t h_{ij}}{N} + 2 h_{il}h_{jl} - Hh_{ij}+ \T_{ij}-\frac12\T  g_{ij}   \\
-\Delta N =\left( \frac{\partial_t H}{N} - |h|^2 - \T_{\nu\nu}-\frac12\T\right)N\, .
\end{cases}
\end{equation}
In the notations of Theorem \ref{teoA} we have 
$$
n=3 \, , \quad \alpha=1 \, , \quad \beta=\gamma=0 \, ,
$$
and 
$$
u=N\,,\quad \mathcal{Q}_{ij}= \frac{\partial_t h_{ij}}{N} + 2 h_{il}h_{jl} - Hh_{ij}+ \T_{ij}-\frac12\T  g_{ij} \,, \quad V=\frac{\partial_t H}{N} - |h|^2 - \T_{\nu\nu}-\frac12\T\, . 
$$
To prove the closure result it is sufficient to show that \eqref{F}-\eqref{hp1}-\eqref{hp2} hold true. First of all \eqref{hp1} holds immediately. Secondly, in order to fulfill also \eqref{hp2} we need 
\begin{equation}\label{k_cond}
1-\sqrt{3}<k<1+\sqrt 3.
\end{equation}
Finally, concerning \eqref{F} we observe the following: by differentiating \eqref{fond} with respect to $t$ we get 
\begin{equation*}
\partial^2_t g_{ij}=-2\partial_t N h_{ij}-2N\partial_t h_{ij}\,,
\end{equation*}
i.e. 
\begin{equation}\label{h_t}
\dfrac{\partial_t h_{ii}}{N}=-\dfrac{\partial^2_t g_{ii}}{2N^2}-\frac{\partial_t N h_{ii}}{N^2}\,. 
\end{equation}
While, tracing \eqref{fond} we get 
\begin{equation}\label{fond_tr}
g^{ij}\partial_t g_{ij}=-2N g^{ij}h_{ij}=-2N H\, ,
\end{equation}
and differentiating \eqref{fond_tr} with respect to $t$, and using 
\begin{equation*}
\partial_t g^{ij}=2N h^{ij}\,,
\end{equation*}
which clearly follows from \eqref{fond}, yields
\begin{equation*}
-4N^2\vert h\vert^2  + g^{ij}\partial_t^2 g_{ij}=-2\partial_t N H - 2N\partial_t H
\end{equation*}
i.e. 
\begin{equation}\label{H_t}
\dfrac{\partial_t H}{N}=2\vert h\vert^2-\dfrac{g^{ij}\partial_t^2g_{ij}}{2N^2}-\dfrac{\partial_tN H}{N^2}\, . 
\end{equation}
Now we can deal with \eqref{F}: for every $i=1,2,3$
\begin{align*}
Q_{ii}+ k V=& \frac{\partial_t h_{ii}}{N} + 2\sum_j \vert h_{ij}\vert^2 - Hh_{ii}+ \T_{ii}-\frac12\T  + k\frac{\partial_t H}{N}-k\vert h\vert^2  -k \T_{\nu\nu}-\frac k2\T \\
 =&-\dfrac{1}{2N^2}\left(\partial^2_t g_{ii}+ kg^{ij}\partial^2_t g_{ij} \right)
 - \frac{\partial_t N}{N^2}\left( h_{ii}+ kH \right) +2\sum_j\vert h_{ij}\vert^2 - Hh_{ii}\\ 
 &\quad +\T_{ii}-\frac{1+k}{2}\T +k\vert h\vert^2-k \T_{\nu\nu}\\
 =&-\dfrac{1}{2N^2}\left(\partial^2_t g_{ii}+2\partial_t N\,h_{ii}\right)-\frac{k}{2N^2}\left( g^{ij}\partial^2_t g_{ij} +2\partial_t N\,H\right)\\
 &\quad+2\sum_j\vert h_{ij}\vert^2 - Hh_{ii}+k\vert h\vert^2\\ 
 &\quad +\T_{ii}-k \T_{\nu\nu}-\frac{1+k}{2}\T \,, 
\end{align*}
where we used \eqref{h_t} and \eqref{H_t}. We diagonalize $h$ and we denote by $\lambda_j$, $j=1,2,3$ its eigenvalues. 

\medskip

\noindent {\em Proof of Theorem \ref{t-main}}: by using the assumption $a)$, Lemma \ref{l-dec} and taking $k\geq 0$, we have
\begin{align*}
-\dfrac{1}{2N^2}\left(\partial^2_t g_{ii}+2\partial_t N\,h_{ii}\right)-\frac{k}{2N^2}\left( g^{ij}\partial^2_t g_{ij} +2\partial_t N\,H\right) &\geq -(1-q)\lambda_i^2-(1-q)k\sum_i \lambda_i^2 \\
&=-(1-q)\lambda_i^2-(1-q)k|h|^2\,.
\end{align*}
Here and in what follows $q=q(t_0)$. Therefore,
\begin{align*}
Q_{ii}+ k V&\geq -(1-q)\lambda_i^2-(1-q)k|h|^2+2\sum_j\vert h_{ij}\vert^2 - Hh_{ii}+k\vert h\vert^2+\T_{ii}-k \T_{\nu\nu}-\frac{1+k}{2}\T \\
&=-(1-q)\lambda_i^2+qk\vert h\vert^2+2\lambda_i^2 - H\lambda_i+\T_{ii}-k \T_{\nu\nu}-\frac{1+k}{2}\T\\
&=(1+q)\lambda_i^2+qk\vert h\vert^2-H\lambda_i+\T_{ii}-k \T_{\nu\nu}-\frac{1+k}{2}\T \, . 
\end{align*}
Now denoting by $\mu_j$, $j=1,2,3$ the eigenvalues of the traceless second fundamental form (and by $e_j$ the corresponding eigenvectors)
$$
\mathring{h}:=h-\frac{H}{3}g \, ,
$$
we have 
\begin{align*}
Q_{ii}+ k V&\geq(1+q)\lambda_i^2+qk\vert h\vert^2-H\lambda_i+\T_{ii}-k \T_{\nu\nu}-\frac{1+k}{2}\T  \\
&= (1+q)\left(\mu_i+\frac{H}{3}\right)^2+qk\vert \mathring{h}\vert^2+\frac{qk}{3}H^2-H\left(\mu_i+\frac{H}{3}\right)+\T_{ii}-k \T_{\nu\nu}-\frac{1+k}{2}\T\\
&\geq \frac{2+2q+3qk}{2}\mu_i^2+\frac{q+3qk-2}{9}H^2+\frac{2q-1}{3}H\mu_i+\T_{ii}-k \T_{\nu\nu}-\frac{1+k}{2}\T\, , 
\end{align*}
for all 
\begin{equation}\label{k_cond_bis}
0\leq k<1+\sqrt 3,
\end{equation}
where we used the fact that   
$$
\vert \mathring{h}\vert^2\geq\frac{3}{2}\mu_i^2\, ,
$$
which holds true, being $\mathring{h}$ trace free, and the fact that $q>1/2>0$. We choose $q=\frac{1}{2}+\eps$, for some $\eps>0$ and $k=1$ (which is coherent with \eqref{k_cond_bis}) in the previous estimate to obtain
\begin{align*}
Q_{ii}+ k V&\geq \left(\frac94+\frac52\eps\right)\mu_i^2 +\frac{4\eps}{9}H^2+\frac{2\eps}{3}H\mu_i+\T_{ii}- \T_{\nu\nu}-\T\\
&\geq \left(\frac94+\frac52\eps-\frac{\eps\theta}{3}\right)\mu_i^2 +\left(\frac{4\eps}{9}-\frac{\eps}{3\theta}\right)H^2+\T_{ii}- \T_{\nu\nu}-\T \\
&= \frac{4\eps(1+\eps)}{9+10\eps}H^2 +\T_{ii}- \T_{\nu\nu}-\T
\end{align*}
for all $i=1,2,3$, where we used Young's inequality with $\theta=\frac{3(9+10\eps)}{4\eps}$. In particular, given a general local orthonormal frame $\{E_i\}$, $i=1,2,3$, on the tangent space of $M^3$ we have
$$
Q_{ii}=Q(e_i,e_i)=c_i^2 Q(E_i,E_i)
$$
for some $c_i\in\RR$ and thus the previous estimate holds for every orthonormal frame $\{E_i\}$. Moreover, choosing the frame of eigenvectors of the stress-energy tensor $T$, we have
$$
\T_{ii}- \T_{\nu\nu}-\T = \T_{ii}-\T_{\nu\nu}-(\T_{ii}+\T_{jj}+\T_{kk}-\T_{\nu\nu})=-\T_{jj}-\T_{kk}=2\Lambda-p_j-p_k
$$
for all $i\neq j\neq k$. Thus
$$
Q_{ii}+ k V\geq \frac{4\eps(1+\eps)}{9+10\eps}H^2+2\Lambda - 2\mathcal{P}(t_0)\geq (n-1)\lambda:=\frac{4\eps(1+\eps)}{9+10\eps} \mathcal{H}({t_0})^2>0,
$$
thanks to $b)$ and $c)$. Hence, we conclude that \eqref{F} holds and so Theorem \ref{teoA} implies that $M^3$ is compact. Computing $A$ and $B$ defined in the proof of Theorem \ref{teoA}, we get
$$
A=\frac83\quad\text{and}\quad B=(n-1)\lambda=\frac{(2q-1)(2q+1)}{10q+4}\mathcal{H}({t_0})^2,
$$
therefore
$$
\mathrm{diam}(M^3,g_{\left.\right|_{t_0}})\leq \frac{\pi}{\mathcal{H}(t_0)}\sqrt{\frac{8(10q+4)}{3(2q-1)(2q+1)}}\,.
$$
Moreover, using the uniformization of closed three-dimensional manifolds prove in \cite{per}, we conclude that $M^3$ must be diffeomorphic to a quotient of $\SS^3$. This concludes the proof of Theorem \ref{t-main}.

\qed

\medskip

\noindent {\em Proof of Theorem \ref{t-main2}}: by using the assumption that $M_{t_0}$ is {\em energy--decelerating} in all directions $a)$, Lemma \ref{l-dec} and taking $k\geq 0$, we have
$$
-\dfrac{1}{2N^2}\left(\partial^2_t g_{ii}+2\partial_t N\,h_{ii}\right)-\frac{k}{2N^2}\left( g^{ij}\partial^2_t g_{ij} +2\partial_t N\,H\right) \geq 0.
$$
Therefore,
\begin{align*}
Q_{ii}+ k V&\geq 2\sum_j\vert h_{ij}\vert^2 - Hh_{ii}+k\vert h\vert^2+\T_{ii}-k \T_{\nu\nu}-\frac{1+k}{2}\T \\
&=2\lambda_i^2 - H\lambda_i+k|h|^2+\T_{ii}-k \T_{\nu\nu}-\frac{1+k}{2}\T\\
&= 2\left(\mu_i+\frac{H}{3}\right)^2-H\left(\mu_i+\frac{H}{3}\right)+k\vert \mathring{h}\vert^2+\frac{k}{3}H^2+\T_{ii}-k \T_{\nu\nu}-\frac{1+k}{2}\T\\
&\geq \left(2+\frac{3k}{2}\right)\mu_i^2+\frac{3k-1}{9}H^2+H\mu_i+\T_{ii}-k \T_{\nu\nu}-\frac{1+k}{2}\T\, , 
\end{align*}
where we used the fact that $\vert \mathring{h}\vert^2\geq\frac 32\mu_i^2$. Arguing as before with Young's inequality and using $b)$, we obtain
$$
Q_{ii}+ k V \geq \frac{18k^2+18k-17}{18(4+3k)}H^2+\T_{ii}-k \T_{\nu\nu}-\frac{1+k}{2}\T \geq \frac{18k^2+18k-17}{18(4+3k)}\mathcal{H}({t_0})^2
$$
for all $i=1,2,3$, if $k> \frac{\sqrt{43}-3}{6}\simeq 0.59$. Thanks to $c)$ , this quantity is uniformly positive. Hence, we conclude that \eqref{F} holds and so Theorem \ref{teoA} implies that $M^3$ is compact. Computing $A$ and $B$ defined in the proof of Theorem \ref{teoA}, we get
$$
A=\frac{2(3+2k-k^2)}{2+2k-k^2}\quad\text{and}\quad B=(n-1)\lambda=\frac{18k^2+18k-17}{18(4+3k)}\mathcal{H}({t_0})^2,
$$
therefore
$$
\mathrm{diam}(M^3,g_{\left.\right|_{t_0}})\leq \frac{6\pi}{\mathcal{H}(t_0)}\sqrt{\frac{(3+2k-k^2)(4+3k)}{(2+2k-k^2)(18k^2+18k-17)}}\,.
$$
Moreover, as before, we conclude that $M^3$ must be diffeomorphic to a quotient of $\SS^3$. This concludes the proof of Theorem \ref{t-main2}.

\end{proof}

\begin{proof}[Proof of Corollary \ref{cor}] 
Consider a perfect fluid spacetime with $q(t_0)>1/2$, $p(t_0)\leq \Lambda$ and $\mathcal{H}(t_0)>0$. Since $p(t,x)=p(t)=\mathcal{P}(t)$ as already observed in the introduction the conclusion follows from Theorem \ref{t-main}.
\end{proof}
\medskip

%
%

\

\subsection*{Acknowledgements} 
The first author is members of the GNSAGA, Gruppo Nazionale per le Strutture Algebriche, Geometriche e le loro Applicazioni of INdAM. The second author is member of GNAMPA, Gruppo Nazionale per l'Analisi Matematica, la Probabilit\`a e le loro Applicazioni of INdAM.

\medskip 
%
%
%
%
%

%
%

\

\

\

\end{document}